\documentclass[10pt]{article}

\usepackage{amsfonts,amsmath,amsthm,amssymb,comment,url} 
\newtheorem{theorem}{Theorem}[section]
\newtheorem{thm}{Theorem}[section]

\newtheorem{lemma}[theorem]{Lemma}
\newtheorem{corollary}[theorem]{Corollary}

\newtheorem{example}{Example}[section]

\def\GF#1{{\mathbb F}_{#1}}

\begin{document}

\title{Packing Costas Arrays}

\author{J.H. Dinitz\\Department of Mathematics and Statistics\\
University of Vermont\\
Burlington, Vermont 05405,
U.S.A. \and P.R.J. \"{O}sterg\r{a}rd\thanks{Research supported in 
part by the Academy of Finland grants 130142 and 132122.}\\Department 
of Communications and Networking\\
Aalto University School of Electrical Engineering\\
P.O. Box 13000\\
00076 Aalto,
Finland \and D.R. Stinson\thanks{Research supported by NSERC grant
203114-06.}\\David R. Cheriton School of Computer Science\\
University of Waterloo\\
Waterloo Ontario, N2L 3G1,
Canada 
}

\date{\today}

\maketitle
\begin{center}
{\large\bf Dedicated to Ralph Stanton}
\end{center}

\begin{abstract}
A Costas latin square of order $n$ is a set of $n$ disjoint
Costas arrays of the same order. Costas latin squares are studied here
from a construction as well as a classification point of view.
A complete classification is carried out up to order 27. 
In this range, we verify the conjecture
that there is no Costas latin square for any odd order $n \geq 3$.
Various other related combinatorial structures are also considered,
including near Costas latin squares (which are certain packings of 
near Costas arrays) and Vatican Costas squares.
\end{abstract}

\section{Introduction}
\label{s1}

A \emph{Costas array} of \emph{order} $n$ (or \emph{side} $n$) 
is an $n \times n$ array of dots and empty cells such that 
\begin{enumerate}
\item there are $n$ dots and $n(n-1)$ empty cells, with
exactly one dot in each row and column, and 
\item
all the segments between pairs of dots differ
in length or in slope.  
\end{enumerate}
Costas arrays were introduced by J.P.\ Costas; see \cite{C,Gol84,GT} for
early results and historical remarks. Extensive surveys of
Costas arrays can be found in \cite{GG,TD}.
%There is a conf. paper from 1975, but Costas apparently studied 
%these already in the 1960s. No need to mention a date(?)

Two Costas arrays of order $n$ are \emph{disjoint} if there is no cell in 
which both arrays have a dot. Let $D(n)$ denote the maximum number of 
mutually disjoint Costas arrays of order $n$. The main topic of this paper
is the study of $D(n)$. Obviously $D(n) \leq n^2/n = n$. We are
particularly interested in cases where this upper bound on $D(n)$ is
attained; in those cases, %the $n \times n$ array is \emph{tiled} by the
%Costas arrays and 
we say that we have a \emph{full set} of mutually
disjoint Costas arrays.

\begin{example}
\label{e1}
We present all the Costas arrays of order $4$. The four arrays
in the first row are mutually disjoint.
%It is interesting to note that these Costas arrays are all 
%in the same orbit under the cyclic or dihedral group.

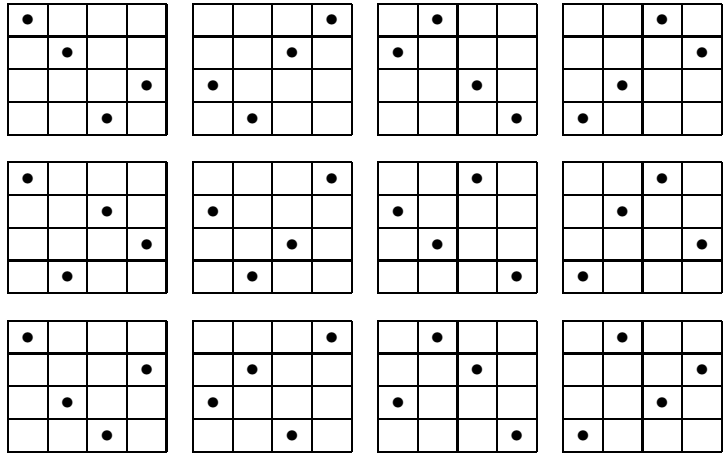
\begin{center}
$
\begin{array}{cccc}
\begin{array}{|c|c|c|c|} \hline
\bullet&&& \\ \hline
&\bullet&&  \\ \hline
&&&\bullet \\  \hline
&&\bullet&  \\  \hline
\end{array}  & %\hspace{.2in}
\begin{array}{|c|c|c|c|} \hline
&&&\bullet \\ \hline
&&\bullet&  \\ \hline
\bullet&&& \\  \hline
&\bullet&&  \\  \hline
\end{array} & % \hspace{.2in}
\begin{array}{|c|c|c|c|} \hline
&\bullet&& \\ \hline
\bullet&&&  \\ \hline
&&\bullet& \\  \hline
&&&\bullet  \\  \hline
\end{array} &  % \hspace{.2in}
\begin{array}{|c|c|c|c|} \hline
&&\bullet& \\ \hline
&&&\bullet  \\ \hline
&\bullet&& \\  \hline
\bullet&&&  \\  \hline
\end{array}\\[.4in]
\begin{array}{|c|c|c|c|} \hline
\bullet&&& \\ \hline
&&\bullet&  \\ \hline
&&&\bullet \\  \hline
&\bullet&&  \\  \hline
\end{array}  & %\hspace{.2in}
\begin{array}{|c|c|c|c|} \hline
&&&\bullet \\ \hline
\bullet&&&  \\ \hline
&&\bullet& \\  \hline
&\bullet&&  \\  \hline
\end{array} & % \hspace{.2in}
\begin{array}{|c|c|c|c|} \hline
&&\bullet& \\ \hline
\bullet&&&  \\ \hline
&\bullet&& \\  \hline
&&&\bullet  \\  \hline
\end{array} &  % \hspace{.2in}
\begin{array}{|c|c|c|c|} \hline
&&\bullet& \\ \hline
&\bullet&&  \\ \hline
&&&\bullet \\  \hline
\bullet&&&  \\  \hline
\end{array} \\[.4in]
\begin{array}{|c|c|c|c|} \hline
\bullet&&& \\ \hline
&&&\bullet  \\ \hline
&\bullet&& \\  \hline
&&\bullet&  \\  \hline
\end{array}  & %\hspace{.2in}
\begin{array}{|c|c|c|c|} \hline
&&&\bullet \\ \hline
&\bullet&&  \\ \hline
\bullet&&& \\  \hline
&&\bullet&  \\  \hline
\end{array} & % \hspace{.2in}
\begin{array}{|c|c|c|c|} \hline
&\bullet&& \\ \hline
&&\bullet&  \\ \hline
\bullet&&& \\  \hline
&&&\bullet  \\  \hline
\end{array} &  % \hspace{.2in}
\begin{array}{|c|c|c|c|} \hline
&\bullet&& \\ \hline
&&&\bullet  \\ \hline
&&\bullet& \\  \hline
\bullet&&&  \\  \hline
\end{array}
\end{array}
$
\end{center} % end tiny
\end{example}

For notational convenience, Costas arrays are often presented using a
certain one-line notation. Given a Costas array 
of order $n$, let $\pi(i) = j$ whenever the array contains a dot in cell $(i,j)$. 
A Costas array of order $n$ can be 
presented as the permutation $\pi = (\pi(1),\pi(2), \cdots   ,\pi(n-1),\pi(n))$.
We call this the \emph{permutation representation} of a Costas array.
Using this notation, the Costas arrays in the first row of Example~\ref{e1} 
have permutation representations $(1,2,4,3)$, $(4,3,1,2)$, $(2,1,3,4)$ and 
$(3,4,2,1)$, respectively.

A \emph{Costas latin square} of order $n$, denoted CLS($n$), is a latin square of 
order $n$ such that for each symbol $i$, $1 \leq i\leq n$,  a Costas array results 
if a dot is placed in the cells containing symbol $i$. Clearly a CLS($n$) is equivalent 
to $n$ disjoint Costas arrays of order $n$.
%, where the dots in each individual Costas array are given the same number in the latin square.
Costas latin squares were first defined and studied by Etzion \cite{E}.

%The example below is a CLS($4$) constructed from the four disjoint Costas 
%arrays of order $4$ given above.

\begin{example}
\label{e2}
The first four disjoint Costas arrays of order\/ $4$ given 
in Example~\ref{e1} are mutually disjoint and lead to
the following\/ $\mathrm{CLS}(4)$.

%{\large
\begin{center}
$\begin{array}{|c|c|c|c|} \hline
1&3&4&2 \\ \hline
3&1&2&4  \\ \hline
2&4&3&1 \\  \hline
4&2&1&3  \\  \hline
\end{array} $
\end{center} %
%}
\end{example}

Etzion \cite{E} defines a \emph{near Costas array} of \emph{order} $n \geq 2$  
to be an $n \times n$ array of dots and empty cells such that 
\begin{enumerate}
\item there are $n-1$ dots and $n^2 - n + 1$ empty cells, with
at most one dot in each row and column, and 
\item 
all the segments between pairs of dots differ
in length or in slope.
\end{enumerate}  
As with Costas arrays, we will say that two near Costas arrays of order $n$ are 
\emph{disjoint} if there is no cell in which both arrays have a dot. 

Let $D_{\mathit{near}}(n)$ denote the maximum number
of near Costas arrays of order $n$. 
%It is easy to see that $D_{\mathit{near}}(2) = 4$. 
Observe that 
$D_{\mathit{near}}(n) \leq \lfloor \frac{n^2}{n-1} \rfloor  = n+1$
for all $n \geq 3$.
%Furthermore, if $D_{\mathit{near}}(n) = n+1$ then the superposition
%of $n+1$ disjoint near Costas arrays has one empty cell (since $n^2 = (n+1)(n-1)+1$). 
%
%\medskip
%\noindent
%\emph{probably change the next definition to require n near Costas's and one Costas or order n}
%\medskip
%
A set of $n+1$ mutually disjoint near Costas arrays of order $n\geq 3$ is called
a \emph{full set} of arrays. In this case, if the filled cells of the $i$th near 
Costas array is given the symbol $i$, then superimposing all these near Costas 
arrays yields an $n \times n$ array on the symbol set $S=\{1,2, \ldots , n+1\}$ 
with exactly one empty cell. Such a square is termed a \emph{near Costas latin square} 
of order $n$, and abbreviated as an NCLS($n$). A somewhat stronger definition could
also be given, where one of the $n+1$ arrays is a Costas array and there is no empty cell.

%satisfying the following properties:
%\begin {enumerate}
%\item there is exactly one empty cell and all others contain exactly one symbol from $S$,
%\item in each row and column each symbol occurs at most once, and
%\item for any fixed symbol $i$ , the cells containing $i$ form a near Costas array 
%of order $n$ (hence each of the $n+1$ symbols occurs exactly $n-1$ times).
%\end{enumerate}

\newpage
\begin{example}
\label{e3}
A near Costas latin square of order 6  {\rm(}a $\mathrm{NCLS}(6)${\rm)}.
%Note that each of the seven symbols occurs exactly five times in the square.

$$
\begin {array} {|c|c|c|c|c|c|}\hline
6&1&7&4&2&3\\ \hline
5&&6&3&1&7\\ \hline
7&2&1&5&3&4\\ \hline
3&5&4&1&6&2\\ \hline
4&6&5&2&7&1\\ \hline
2&4&3&7&5&6\\ \hline
\end{array}
$$
\end{example}

The roles of the rows, columns and symbols of a latin square can be
interchanged. By interchanging the roles of rows and symbols, a Costas latin square of order $n$
can be transformed into an (equivalent) latin square $L = (l_{i,j})$ with the property 
that for any $1 \leq k \leq n-1$, 
$l_{i,j+k}-l_{i,j} = l_{i,m+k}-l_{i,m}$ implies that $j=m$. We call
such a square a \emph{row-Costas latin square}. Analogously, 
interchanging the roles of columns and symbols, we get
the definition of \emph{column-Costas latin squares}.
%\wnote{Better names? Did somebody else mention or consider these
%earlier?} 

We define one final related concept. A {\em Vatican square} of order $n$ is a latin square
$L = (l_{i,j})$ of order $n$ with the property that for any fixed $d$, $1 \leq d \leq n-1$,
all ordered pairs of the form $(l_{i,j},l_{i,j+d})$, $1 \leq i \leq n$, 
$1 \leq j \leq n-d$ are distinct. A \emph{Vatican Costas square} is a latin square that is both a Vatican square and a Costas latin square.   Vatican 
squares were introduced in \cite{GT2} while Vatican Costas squares
were first considered by Etzion in \cite{E}.  In a Vatican Costas square one can interchange
the roles of the rows, columns, and symbols to get 
other equivalent structures.

%The general definition of Vatican 
%squares was introduced in \cite{GT2}. A \emph{Vatican Costas square}
%of order $n$, considered by Etzion in \cite{E}, is a Costas latin square 
%$L = (l_{i,j})$ of order $n$ with the property that for any fixed $d$, $1 \leq d \leq n-1$,
%all ordered pairs of the form $(l_{i,j},l_{i,j+k})$, $1 \leq n$, 
%$1 \leq j \leq n-d$ are distinct. Also in this case one can interchange
%the roles of the rows, columns, and symbols to get various versions.

%We define one final related concept. The general definition of Vatican 
%squares was introduced in \cite{GT2}. A \emph{Vatican Costas square}
%of order $n$, considered by Etzion in \cite{E}, is a Costas latin square 
%$L = (l_{i,j})$ of order $n$ with the property that for any fixed $d$, $1 \leq d \leq n-1$,
%all ordered pairs of the form $(l_{i,j},l_{i,j+k})$, $1 \leq n$, 
%$1 \leq j \leq n-d$ are distinct. Also in this case one can interchange
%the roles of the rows, columns, and symbols to get various versions.

The paper is organized as follows. In Section~\ref{sect:constr}
some basic constructions for Costas arrays are surveyed and
applied to the problem of constructing Costas latin squares;
near Costas arrays and near Costas latin squares are also
considered. A computational study is carried out in 
Section~\ref{sect:enum} to enumerate Costas latin squares of
small order; also other related structures defined above are
considered.

\section{Constructions}

\label{sect:constr}

\subsection{Costas Latin Squares}

%\bigskip
One way to construct Costas latin squares would seem to be to start 
with an initial Costas array of order $n$ and then move each of the 
dots to the right $k$ cells (mod $n$) to construct the $k$th array 
(where $0\leq k\leq n-1$).  Certainly these arrays will be mutually 
disjoint, but in general they will not be Costas arrays. However, this 
technique will succeed provided that we start with a singly periodic 
Costas array, which we now define.

Let an $n \times n$ array be \emph{left-right extended} by putting a dot in cell
$(i,j+n)$ if and only if there is a dot in cell $(i,j)$.
A \emph{singly periodic Costas array} of order $n$ is  an $n\times
n$ array such that every $n \times n$ window in its left-right extension is
a Costas array of order $n$. The following result is now obvious.

\begin{thm}  
\label{thm1}
If there exists a singly periodic Costas array 
of order\/ $n$, then there exists a\/
$\mathrm{CLS}(n)$.
\end{thm}

Next, we describe the well known \emph{Welch construction} for Costas arrays.
Let $p$ be prime and $\alpha $ be 
a primitive element in the field $\GF{p}$.  Let $n=p-1$.
A Costas array of order $n$ is obtained
by placing a dot at $(i,j)$ if and only if $i=\alpha ^j$,
for $0 \leq j < n$, and
$i=1,\dots, n$.  We get the following result,
which is implicit in \cite{GET} and explicit in \cite{E}.

\begin{corollary} \label{cor2} 
There exists a\/ $\mathrm{CLS}(p-1)$
for all primes $p$.
\end{corollary}

\begin{proof}
It is known that Costas arrays constructed 
via the Welch construction are singly periodic.
Apply Theorem \ref{thm1}.
\end{proof}

\noindent
\emph{Remark.} In \cite{GET}, this method was used to
construct a Vatican square.

%\medskip
%\noindent
%\emph{note -- this may already be known.  See theorem D in \cite{GET}}

\bigskip
Using 2 for the primitive element in $\GF{5}$, one obtains the following 
Costas array of order $4$, given on the left below. (Note here that the columns 
are labeled $0,1,2,3,$ while the rows are labeled $1,2,3,4$ from the bottom.)  
This Costas array is singly periodic. On the right is the CLS($4$) obtained 
from this array. It is interesting to note that this Costas latin square 
is different from any dihedral transformation of the CLS($4$) presented 
in Example \ref{e1} (e.g., consider the front and back diagonals). 

\begin{center}
$\begin{array}{|c|c|c|c|} \hline
&&\bullet& \\ \hline
&&&\bullet  \\ \hline
&\bullet&& \\  \hline
\bullet&&&  \\  \hline 
\end{array} $ \hspace{.8in}
$\begin{array}{|c|c|c|c|} \hline
3&4&1&2 \\ \hline
2&3&4&1  \\ \hline
4&1&2&3 \\  \hline
1&2&3&4  \\  \hline  
\end{array} $ 
\end{center}

Unfortunately, the only known singly periodic Costas arrays of order $n$ come
from the Welch construction when $n=p-1$. It is also known that 
a singly periodic Costas array of order $n$ does not exist if
$n>1$ is odd, or if $n \in \{8, 14, 20 \}$. %(get an original reference) 

Corollary \ref{cor2} gives an infinite class of CLS($n$).
Etzion \cite{E} gave another infinite class which contains no
new orders but which are inequivalent to the previous class. 
Etzion's second infinite class is based on the {\it Golomb construction},
which we now describe.  

Let $\alpha$ and $\beta$ be primitive elements
in a finite field $\mathbb{F}_q$. Denote $n = q-2$. For $1 \leq i \leq n$,
$1 \leq j \leq n$, place a dot in cell $(i,j)$ of an
$n \times n$ array if and only if $\alpha^i + \beta^j = 1$.
Denote the resulting array by $G(\alpha, \beta)$.
Then $G(\alpha, \beta)$ is a Costas array of order $n$.

Etzion's infinite class is constructed as follows.
Suppose that $q = 2^r$ and suppose that
$q-1$ is a (Mersenne) prime. Let $\alpha \in \mathbb{F}_q$ be
a primitive element. Note that $\alpha^i$ is primitive for $1 \leq i \leq n$
because $q-1$ is prime. It is not difficult to verify that the
$n$ Costas arrays $G(\alpha, \alpha^i)$ are disjoint, for
$1 \leq i \leq n$. This yields a CLS($2^r-2$) whenever $2^r-1$ is
a Mersenne prime.

\subsection{Near Costas Latin Squares}

%I thought it might be interesting to play around with the
%Golomb construction to see when disjoint Costas or
%near-Costas arrays could be constructed. Here are some
%observations.
%
%\subsection{Applications of  the Golomb Construction}

We begin with a simple observation.

\begin{lemma} Suppose there exists a $\mathrm{CLS}(n)$.
Then $D_{\mathit{near}}(n) \geq 
n$.
\end{lemma}

\begin{proof}
Given a $\mathrm{CLS}(n)$, we have $n$ disjoint
Costas arrays of order $n$. Remove one dot from each of 
these $n$ arrays, yielding $n$ disjoint near-Costas arrays
of order $n$.
\end{proof}

\begin{corollary}
\label{NCLS.cor}
If $p$ is a prime, then 
$D_{\mathit{near}}(p-1) \geq 
p-1$.
\end{corollary}

We now provide another way of obtaining $n-1$ near-Costas arrays
of order $n-1$. This method works whenever $n$ is a prime power;
it is based on the {\it Lempel construction}
which is the special case of the Golomb construction where $\beta = \alpha$.
We will use the following slight variation:
Let $q$ be a prime power and label the rows and columns of $q-1$ 
square arrays of side $q-1$ by the integers $0, \dots , q-2$.
The arrays will be named as $A_c$, $c  \in \mathbb{F}_q^*$.
For each such array $A_c$, place a dot in cell $(i,j)$ whenever $\alpha^i +\alpha ^j = c$. 
 Clearly these $q-1$ arrays are all disjoint and symmetric. Each array
 contains $q-2$ dots which miss exactly one row and column
(namely, row and column $x$, where $\alpha ^x = c$).  

We now show 
that each of these arrays is a near-Costas array. This can be done by
modifying the proof of 
\cite[Theorem 2]{Gol84}. Suppose that
$A_c$ contains dots in cells $(i,j)$, $(i+u,j+v)$, $(i',j')$ and $(i'+u,j'+v)$,
where $i \neq i'$, $j \neq j'$ and $u,v \neq 0$.
Then
\[ \alpha^i + \alpha^j = \alpha^{i+u} + \alpha^{j+v} =
\alpha^{i'} + \alpha^{j'} = \alpha^{i'+u} + \alpha^{j'+v} = c.\]
We then obtain
\[ \alpha^{i+u} + \alpha^v(c - \alpha^i) = \alpha^{i'+u} + \alpha^v(c - \alpha^{i'}) = c,\]
and therefore
\[ \alpha^{i+u} - \alpha^{i+v} = \alpha^{i'+u} - \alpha^{i'+v} = c(1 - \alpha^v).\]
We have $c \neq 0$, and $\alpha^v \neq 1$ follows because $v \neq 0$.
Hence, 
\[ \alpha^i(\alpha^{u} - \alpha^{v}) = \alpha^{i'}(\alpha^{u} - \alpha^{v}) \neq 0.\]
Notice that $u \neq v$ because $\alpha^i(\alpha^{u} - \alpha^{v}) \neq 0$.
However, since $u \neq v$, we obtain $\alpha^i = \alpha^{i'}$, which is a contradiction 
because $i \neq i'$. 

Summarising the above discussion, we have the following 
theorem.
\begin{theorem}
\label{NCLS.thm}
If $q$ is a prime power, then 
$D_{\mathit{near}}(q-1) \geq 
q-1$.
\end{theorem}

\noindent\emph{Remark.}  Theorem \ref{NCLS.thm} is more general than
Corollary \ref{NCLS.cor} in that it applies to prime powers
rather than just primes.

\begin{example}
Let $q = 7$ and take $\alpha = 3$. We present the superposition of
$A_1, \dots, A_6$: 

\[
\begin{array}{|c|c|c|c|c|c|}
 \hline
2 & 4 & 3 &  & 5 & 6 \\ \hline
4 & 6 & 5 & 2 &  & 1 \\ \hline
3 & 5 & 4 & 1 & 6 &  \\ \hline
 & 2 & 1 & 5 & 3 & 4 \\ \hline
5 &  & 6 & 3 & 1 & 2 \\ \hline
6 & 1 &  & 4 & 2 & 3 \\ \hline
\end{array}
\]
\end{example}

Next, we use the Golomb construction 
to construct disjoint near Costas
arrays of order $q-2$, for certain prime powers $q \equiv 3 \bmod 4$.
Since $q \equiv 3 \bmod 4$ is a prime power, $-1$ is a
quadratic non-residue modulo $q$. Suppose also 
that $q = 2p+1$ where $p$ is an odd prime. 
Then $\mathbb{F}_q$ has $\phi(2p) = p-1$
primitive roots, say $\alpha_k$, $k = 1, \dots , p-1$.

Denote $\gamma = \alpha_k$, $\delta = \alpha_{\ell}$
and $\epsilon = \alpha_m$, where $\ell \neq m$.
We will show that $G(\gamma, \delta)$ 
and $G(\gamma, \epsilon)$ contain exactly one
common dot.
Suppose to the contrary that $\gamma^i + \delta^j = \gamma^i + \epsilon^j = 1$.
Then  $\delta^j = \epsilon^j$. Denote $\zeta = \delta / \epsilon$;
then $\zeta^j = 1$. It is clear that $\mathsf{ord}(\zeta) \mid 2p$,
so  $\mathsf{ord}(\zeta) \in \{1,2,p,2p\}$. We consider each possibility
in turn.

\begin{description}
\item [case 1: $\mathsf{ord}(\zeta) = 1$.]  Here $\zeta  = 1$,
so $\delta = \epsilon$, which is a contradiction.

\item [case 2: $\mathsf{ord}(\zeta) = 2$.]
  Here $\zeta  = -1$ 
and $\delta = -\epsilon$. However, $-1$ is a quadratic non-residue,
so $\delta$ and $\epsilon$ cannot both be primitive elements.
This is a contradiction.

\item [case 3: $\mathsf{ord}(\zeta) = p$.]
Here $\zeta$ is a quadratic residue and
$\delta ^j = \epsilon ^j = -1$. This implies that $j = p$
and $i$ is the unique element such that $\gamma^i = 2$.

\item [case 4: $\mathsf{ord}(\zeta) = 2p$.]
Here $\zeta$ is a primitive element 
and $\delta ^j = \epsilon ^j =  1$. This implies that $j = 2p$, which
is a contradiction because $j \leq 2p-1$.
\end{description}

It follows from this discussion that 
the $p-1$ Costas arrays $G(\alpha_1, \alpha_i)$ ($1 \leq i \leq p-1$)
all contain a single common dot. On removing this dot, we obtain 
$p-1$ disjoint near-Costas arrays of order $2p-1$.
Summarising, we have the following.

\begin{theorem}
\label{thm:near}
Suppose that $q \equiv 3 \bmod 4$ is a prime power and suppose 
that $(q-1)/2$ is prime. Then $D_{\mathit{near}}(q-2) \geq 
(q-3)/2$.
\end{theorem}

Theorem \ref{thm:near} is illustrated by the following example.

\begin{example}
Let $q = 11$, so $n = 9$ and
$p = 5$. The four primitive elements modulo $11$ are 
$2$, $8$, $7$ and $6$. We present the superposition of
$G(2, 2)$, $G(2, 8)$, $G(2, 7)$, $G(2, 6)$. The unique common dot
is in position $(1,5)$.
If this dot is deleted, then we get 
four disjoint near Costas arrays of order $9$. 

\[
\begin{array}{|c|c|c|c|c|c|c|c|c|}
 \hline
 & & & & \bullet & & & & \\ \hline
8 & & 2 & & & &  6 & & 7 \\ \hline
& 2 && 8 && 7 && 6 & \\ \hline
7 && 6 &&&& 2 && 8 \\ \hline
2 && 7 &&&& 8 && 6 \\ \hline
& 6 && 7 && 8 && 2 & \\ \hline
& 7 && 2 && 6 && 8 & \\ \hline
& 8 && 6 && 2 && 7& \\ \hline
6 && 8 &&&& 7 && 2  \\ \hline
\end{array}
\]
\end{example}

\noindent\emph{Remark.}
It may be the case that various Costas arrays of the form $G(\alpha, \alpha)$ are disjoint.
When $q = 11$, one finds that the four Costas arrays $G(2, 2)$, $G(8, 8)$, $G(7, 7)$
and $G(6, 6)$ are disjoint. The resulting array is presented in the next example.

\newpage
\begin{example}The four disjoint Costas arrays $G(2, 2)$, $G(8, 8)$, $G(7, 7)$
and $G(6, 6)$.
\[
\begin{array}{|c|c|c|c|c|c|c|c|c|}
 \hline
6 & 7 &   & 8 & 2 &   &   &   &   \\ \hline
7 &   & 2 & 6 &   & 8 &   &   &   \\ \hline
  & 2 & 8 &   & 7 & 6 &   &   &   \\ \hline
8 & 6 &   &   &   &   & 2 & 7 &   \\ \hline
2 &   & 7 &   &   &   & 8 &   & 6 \\ \hline
  & 8 & 6 &   &   &   &   & 2 & 7 \\ \hline
  &   &   & 2 & 8 &   & 7 & 6 &   \\ \hline
  &   &   & 7 &   & 2 & 6 &   & 8 \\ \hline
  &   &   &   & 6 & 7 &   & 8 & 2  \\ \hline
\end{array}
\]
\end{example}

It is not obvious how to derive a simple algebraic formula
to determine when Costas arrays of the form $G(\alpha, \alpha)$ are disjoint.
So at present we have no theorem that generalises the previous example.

\section{Classification of Squares}

\label{sect:enum}

The concept of equivalence (or isomorphism) is central in any
classification of combinatorial structures.
Two Costas arrays are said to be \emph{equivalent} if there
is a mapping in the dihedral group of rotations and reflections 
of the square (which has order 8) that maps one array onto
the other. Two Costas latin squares are \emph{equivalent} if there
is a mapping in the same group that maps every Costas array
of one square onto a Costas array of the other. (Note that with
this definition, the numbers given to the Costas arrays in the
Costas latin square are irrelevant.)

The set of all mappings that map a Costas array (or Costas latin square)
onto itself forms the \emph{automorphism group} of the respective structure.
The definitions of equivalence and
automorphism groups for near Costas arrays and near Costas 
latin squares are analogous. 

\subsection{Costas Latin Squares}

Let us first have a brief look at Costas latin squares with
very small parameters.
It is trivial to construct CLS($1$) and CLS($2)$; these are
unique up to equivalence.
There is no CLS($3$) since no Costas array of order $3$ can contain 
a dot in the middle cell. However, it is easy to construct
two disjoint Costas arrays of order 3, so it follows that $D(3) = 2$.

A CLS($4$) is presented in Example \ref{e2}. In the next case, $n=5$, 
it turns out that there is no Costas latin square of order 5. This can 
still be done by hand (for example, with a case-by-case argument). 
The following example shows that there are four disjoint Costas arrays 
of order 5, so it follows that $D(5) = 4$.

\begin{example}
Here are four disjoint Costas arrays of order $5$,
given in permutation representation:
\begin{center}$(1,4,5,3,2)$, $(2,5,3,4,1)$,
$(3,1,2,5,4)$ and $(4,3,1,2,5)$.\end{center}

%\begin{center} $
%\begin{array}{|c|c|c|c|c|} \hline
%1&4&5&3&2 \\ \hline
%\end{array} \hspace{.2in}
%\begin{array}{|c|c|c|c|c|} \hline
%2&5&3&4&1 \\            \hline
%\end{array} \hspace{.2in}
%\begin{array}{|c|c|c|c|c|} \hline
%3&1&2&5&4 \\      \hline
%\end{array} \hspace{.2in}
%\begin{array}{|c|c|c|c|c|} \hline
%4&3&1&2&5 \\      \hline
%\end{array}$
%\end{center}

\noindent
It is interesting to note that a $5$th Costas array which is disjoint to the 
other four arrays in three cells also exists. This Costas array has 
permutation representation \[(5,2,4,3,1).\]
%\begin{array}{|c|c|c|c|c|} \hline
%5&2&4&3&1 \\ \hline
%\end{array} $
If we superimpose these five Costas arrays, we get the following array.

\begin{center}
$\begin{array}{|c|c|c|c|c|} \hline
1&2&3&4&5 \\ \hline
3&5&4&1&2  \\ \hline
4&3&2&5&1 \\  \hline
 &4&1,5&2&3  \\  \hline
2,5&1& &3&4 \\  \hline
\end{array} $
\end{center} %
\end{example}

We know of no better way than a computer search to handle
orders greater than $5$. For a given order $n$, the largest 
possible number of mutually disjoint Costas arrays can be found by
solving the following instance of the maximum clique problem.
Form a graph $G$ with one vertex for each Costas array and with
an edge between two vertices exactly when the corresponding arrays
are disjoint. Any clique in $G$ corresponds to a set of mutually
disjoint Costas arrays, so we may now search for the size of a
maximum clique. The \emph{Cliquer} software \cite{NO} was used
to search for (maximum) cliques
in the current work. All Costas arrays of order between
1 and 27 can be obtained in electronic form from \cite{Costas}.  

Whenever there are $n$ disjoint Costas arrays of order $n$, 
one can find all cliques of size $n$ in the above mentioned graph
to get all possible Costas latin squares. However, utilizing the
fact that the arrays then cover all the cells of the square, such cases are
for performance reasons preferably considered as instances of
the exact cover problem. In the \emph{exact cover problem},
we are given a set $S$ and a collection $U$ of its subsets, and the
task is to form a partition of $S$ by using sets in $U$. Here,
$S$ is the set of all cells of the square, and $U$ has one
set for each Costas array (consisting of the cells where there
are dots). The \textsf{libexact} library \cite{KP} 
was used here to solve instances of the exact cover problem.

The computational results for all orders up to 27 are summarized in 
Table \ref{tab:c}. The columns in the table are the order $n$,
the total number of Costas arrays Nc, the number of equivalence
classes of Costas arrays Nce,
the largest number of disjoint Costas arrays $D(n)$, and, if a
full set of disjoint Costas arrays exists, the total number of Costas latin
squares Nl and the number of equivalence classes of Costas latin
squares Nle. The entries for Nc and Nce have been taken
from \cite{D,GG,TD}.

\renewcommand{\tabcolsep}{11pt}

\begin{table}[htbp]
\caption{Computational results for Costas arrays}
\label{tab:c}
\begin{center}
\begin{tabular}{|crrrrr|}\hline
$n$& Nc   & Nce & $D(n)$& Nl & Nle\\\hline
1  & 1 	  & 1   &  1 & 1  & 1  \\
2  & 2 	  & 1   &  2 & 1  & 1\\
3  & 4 	  & 1   &  2 &    &\\
4  & 12   & 2   &  4 & 7  & 3\\
5  & 40   & 6   &  4 &    &\\
6  & 116  & 17  &  6 & 124& 26\\
7  & 200  & 30  &  6 &    &\\
8  & 444  & 60  &  8 & 312& 85\\
9  & 760  & 100 &  8 &    &\\
10 & 2160 & 277 & 10 & 128& 30\\
11 & 4368 & 555 & 10 &    &\\
12 & 7852 & 990 & 12 & 16346 & 3761\\
13 & 12828& 1616& $\leq 12$  &    &\\
14 & 17252& 2168& $\leq 13$  &    &\\
15 & 19612& 2467& $\leq 14$  &    &\\
16 & 21104& 2648& 16 & 32768 & 8256\\
17 & 18276& 2294& $\leq 16$  &    &\\
18 & 15096& 1892& 18 & 5832 & 756\\
19 & 10240& 1283& $\leq 18$  &    &\\
20 & 6464 & 810 & $\leq 19$  &    &\\
21 & 3536 & 446 & 11 &    &\\
22 & 2052 & 259 & 22 & 200& 30\\
23 & 872  & 114 &  9 &    &\\
24 & 200  & 25  &  8 &    &\\
25 & 88   & 12  &  5 &    &\\
26 & 56   & 8   &  6 &    &\\
27 & 204  & 29  &  8 &    &\\\hline
\end{tabular}
\end{center}
\end{table}

\renewcommand{\tabcolsep}{6pt}

It is easy to determine the automorphism groups of the classified
Costas latin squares. The results were then validated by applying
the orbit--stabilizer theorem to get the total number, which 
coincides with the number in the column Nl of Table~\ref{tab:c}.

Note that there exists a CLS($8$); this was also observed in \cite{E}.  
This is the only order $n$ for which there exists a CLS($n$) but $n+1$ is 
not a prime number. An example of a CLS(8) can be found in the Appendix.

Table~\ref{tab:c} lends additional numerical evidence to the truth of 
the conjecture of Etzion \cite[Conjecture 2]{E} that there is no CLS($n$) for any 
odd $n \geq 3$.

%\bigskip
%\emph{Do we want to list one of each order (a nonglide one) in an appendix??}

\vspace*{-.1in}
\subsection{Disjoint Near Costas Arrays}
\vspace*{-.1in}

For near Costas arrays, one can carry out a similar computational 
study as for Costas arrays in the previous subsection.
In this case exact cover cannot be used in a direct way to find full 
sets of disjoint arrays. One possibility is to use a clique search
to find full sets, another is to use the framework of exact cover
for all $n^2$ possibilities for the empty cell in the near Costas 
latin square.

In any case, the number of near Costas latin squares grows very
quickly; already for order 5 this number is 978982. Consequently,
we will not be able to present an extensive table like that in
Table~\ref{tab:c}. Instead, in the Appendix, we give just an example of a near
Costas latin square for each order up to $n=8$.

%\begin{table}[htbp]
%\caption{Computational results for near Costas arrays}
%\label{tab:nc}
%\begin{center}
%\begin{tabular}{crrrrr}\hline
%$n$& Nc   & Nce & $D(n)$& Nl & Nle\\\hline
%1  & 1 	  &     &  1 &    &  \\
%2  & 4 	  &     &  3 & 4  &  \\
%3  & 18   &     &  4 & 72 &\\
%4  & 88   &     &  5 & 5288 &  \\
%5  & 420  &     &  6 & 978982 &\\
%6  & 2112 &     &  7 &    & \\
%7  & 9844 &     &    &    &\\
%8  & 43920&     &    &    & \\\hline
%\end{tabular}
%\end{center}
%\end{table}

\vspace*{-.1in}
\subsection{Vatican Costas Squares}
\vspace*{-.1in}

For some of the structures defined in the introduction, the mappings
in the definitions of equivalence must map rows to rows 
(resp.\ columns to columns), that is, the mapping is in a 
dihedral group of order 4 rather than 8. Vatican Costas squares, which
we consider here, are such structures (so this specification should
be added to the definition on \mbox{\cite[p.\ 149]{E}).}

We consider one representative from each equivalence class of Costas 
latin squares classified in the current work (cf.\ column Nle in
Table~\ref{tab:c}), and check whether it is a Vatican Costas square.
This check is also carried out for the square obtained by a 90 degree 
rotation. For the orders up to 27 for which Costas latin squares 
exist, that is, $n = 2$, 4, 6, 8, 10, 12, 16, 18 and 22, the 
number of equivalence classes of Vatican Costas squares is
1, 1, 3, 0, 2, 2, 4, 3 and 5, respectively. One square from each equivalence class is given in the Appendix. Etzion \cite{E} already
classified such squares of order 6.
It is interesting to note that many of these squares have the property that all of the columns are translates of a single column. Also note that $n=8$ is the only order for which there is a Costas latin square but no Vatican Costas square
(coincidentally, perhaps, it is also the only order
$n$ where there is a Costas latin square but $n+1$ is not a prime).

 %; cf.\ \cite[Conjecture 3]{E}.
%\wnote{Display these squares? Study them to find new constructions?}

\newpage

\section*{Appendix}

\subsection*{A.1 \ A Costas latin square of side 8}

\[
\begin{array}{|c|c|c|c|c|c|c|c|}
 \hline
1  & 5  & 7  & 4  & 2  & 8  & 3  & 6  \\ \hline
2  & 4  & 3  & 6  & 7  & 5  & 1  & 8  \\ \hline
3  & 6  & 1  & 8  & 5  & 4  & 2  & 7  \\ \hline
6  & 3  & 8  & 2  & 4  & 7  & 5  & 1  \\ \hline
8  & 1  & 5  & 7  & 6  & 3  & 4  & 2  \\ \hline
4  & 8  & 2  & 1  & 3  & 6  & 7  & 5  \\ \hline
5  & 7  & 6  & 3  & 1  & 2  & 8  & 4  \\ \hline
7  & 2  & 4  & 5  & 8  & 1  & 6  & 3   \\ \hline
\end{array}
\]

\subsection*{A.2 \  Near Costas latin squares of orders 2 through 8}

\renewcommand {\arraycolsep}{5pt}
%\begin{figure}[htbp]

$$
\begin{array}{|c|c|}
%\multicolumn{2}{c}{}\\
%\multicolumn{2}{c}{}\\
%\multicolumn{2}{c}{}\\
%\multicolumn{2}{c}{}\\
\hline
   & 2 \\\hline
 1 & 3 \\\hline
\end{array}
\hspace{0.2in}
\begin{array}[c]{|c|c|c|}
%\multicolumn{3}{c}{}\\
%\multicolumn{3}{c}{}\\
%\multicolumn{3}{c}{}\\
\hline
   & 2 & 1 \\\hline
 4 & 1 & 3 \\\hline
 3 & 4 & 2 \\\hline
\end{array}
\hspace{0.2in}
\begin{array}[c]{|c|c|c|c|}
%\multicolumn{4}{c}{}\\
%\multicolumn{4}{c}{}\\
\hline
   & 2 & 1 & 3 \\\hline
 4 & 5 & 2 & 1 \\\hline
 3 & 1 & 5 & 4 \\\hline
 5 & 4 & 3 & 2 \\\hline
\end{array}
\hspace{0.2in}
\begin{array}[c]{|c|c|c|c|c|}
%\multicolumn{5}{c}{}\\
\hline
   & 1 & 3 & 5 & 2 \\\hline
 4 & 2 & 6 & 1 & 3 \\\hline
 6 & 5 & 1 & 2 & 4 \\\hline
 5 & 4 & 2 & 3 & 6 \\\hline
 3 & 6 & 5 & 4 & 1 \\\hline
\end{array}
$$

$$
\begin{array}[c]{|c|c|c|c|c|c|}
%\multicolumn{6}{c}{}\\
\hline
   & 1 & 3 & 4 & 7 & 5 \\\hline
 6 & 2 & 7 & 1 & 3 & 4 \\\hline
 5 & 4 & 6 & 2 & 1 & 3 \\\hline
 7 & 5 & 1 & 3 & 2 & 6 \\\hline
 4 & 7 & 2 & 5 & 6 & 1 \\\hline
 3 & 6 & 5 & 7 & 4 & 2 \\\hline
\end{array}
\hspace{0.7in} \begin{array}[c]{|c|c|c|c|c|c|c|}
\hline
   & 1 & 2 & 3 & 4 & 6 & 5 \\\hline
 3 & 2 & 1 & 8 & 5 & 7 & 6 \\\hline
 4 & 5 & 8 & 7 & 1 & 2 & 3 \\\hline
 7 & 4 & 6 & 2 & 3 & 5 & 8 \\\hline
 5 & 6 & 7 & 1 & 8 & 3 & 4 \\\hline
 8 & 7 & 5 & 6 & 2 & 4 & 1 \\\hline
 6 & 8 & 3 & 4 & 7 & 1 & 2 \\\hline
\end{array}
$$

$$
%\\[.7in]
%\multicolumn{4}{l}{
\begin{array}[c]{|c|c|c|c|c|c|c|c|}
\hline
   & 1 & 2 & 4 & 5 & 7 & 9 & 3 \\\hline
 3 & 2 & 1 & 8 & 9 & 6 & 4 & 5 \\\hline
 4 & 9 & 6 & 3 & 8 & 1 & 5 & 7 \\\hline
 9 & 7 & 3 & 1 & 2 & 4 & 8 & 6 \\\hline
 5 & 8 & 7 & 6 & 4 & 9 & 2 & 1 \\\hline
 8 & 4 & 5 & 2 & 7 & 3 & 6 & 9 \\\hline
 7 & 6 & 9 & 5 & 1 & 8 & 3 & 2 \\\hline
 6 & 5 & 8 & 7 & 3 & 2 & 1 & 4 \\\hline
\end{array}
$$
\\

\subsection*{A.3 \ Vatican Costas  squares}

\vspace{-.05in}
\noindent
$n=2$ \[ 
\begin{array}{|c|c|}
\hline
 1 & 2 \\\hline
 2 & 1 \\\hline
\end{array}
\]

\vspace{-.05in}
\noindent
$n=4$
\[
\begin{array}{|c|c|c|c|}
\hline
 1 & 2 & 4 & 3 \\\hline
 2 & 3 & 1 & 4 \\\hline
 3 & 4 & 2 & 1 \\\hline
 4 & 1 & 3 & 2 \\\hline
\end{array}
\]

\vspace{-.05in}
\noindent
$n=6$
\[
\begin{array}{|c|c|c|c|c|c|}
\hline
1&2&3&4&5&6\\\hline
3&6&2&5&1&4\\\hline
5&3&1&6&4&2\\\hline
2&4&6&1&3&5\\\hline
6&5&4&3&2&1\\\hline
4&1&5&2&6&3\\\hline
\end{array}\hspace{.2in}
\begin{array}{|c|c|c|c|c|c|}
\hline
1&3&2&5&6&4\\\hline
2&4&3&6&1&5\\\hline
3&5&4&1&2&6\\\hline
4&6&5&2&3&1\\\hline
5&1&6&3&4&2\\\hline
6&2&1&4&5&3\\\hline
\end{array}\hspace{.2in}
\begin{array}{|c|c|c|c|c|c|}
\hline
1&2&3&4&5&6\\\hline
4&1&5&2&6&3\\\hline
3&6&2&5&1&4\\\hline
2&4&6&1&3&5\\\hline
6&5&4&3&2&1\\\hline
5&3&1&6&4&2\\\hline
\end{array}
\]

\renewcommand{\arraycolsep}{2pt}

\vspace{-.05in}
\noindent
$n=10$
\[
\begin{array}{|c|c|c|c|c|c|c|c|c|c|}
\hline
1&2&9&3&5&10&8&4&7&6\\\hline
2&3&10&4&6&1&9&5&8&7\\\hline
3&4&1&5&7&2&10&6&9&8\\\hline
4&5&2&6&8&3&1&7&10&9\\\hline
5&6&3&7&9&4&2&8&1&10\\\hline
6&7&4&8&10&5&3&9&2&1\\\hline
7&8&5&9&1&6&4&10&3&2\\\hline
8&9&6&10&2&7&5&1&4&3\\\hline
9&10&7&1&3&8&6&2&5&4\\\hline
10&1&8&2&4&9&7&3&6&5\\\hline
\end{array}\hspace{.3in}
\begin{array}{|c|c|c|c|c|c|c|c|c|c|}\hline
1&4&5&7&3&8&2&10&9&6\\\hline
2&5&6&8&4&9&3&1&10&7\\\hline
3&6&7&9&5&10&4&2&1&8\\\hline
4&7&8&10&6&1&5&3&2&9\\\hline
5&8&9&1&7&2&6&4&3&10\\\hline
6&9&10&2&8&3&7&5&4&1\\\hline
7&10&1&3&9&4&8&6&5&2\\\hline
8&1&2&4&10&5&9&7&6&3\\\hline
9&2&3&5&1&6&10&8&7&4\\\hline
10&3&4&6&2&7&1&9&8&5\\\hline
\end{array}\]

\renewcommand{\arraycolsep}{1.2pt}

\small
\noindent
{\normalsize $n=12$}
\[
\begin{array}{|c|c|c|c|c|c|c|c|c|c|c|c|}
\hline
1&2&5&3&10&6&12&4&9&11&8&7\\\hline
2&3&6&4&11&7&1&5&10&12&9&8\\\hline
3&4&7&5&12&8&2&6&11&1&10&9\\\hline
4&5&8&6&1&9&3&7&12&2&11&10\\\hline
5&6&9&7&2&10&4&8&1&3&12&11\\\hline
6&7&10&8&3&11&5&9&2&4&1&12\\\hline
7&8&11&9&4&12&6&10&3&5&2&1\\\hline
8&9&12&10&5&1&7&11&4&6&3&2\\\hline
9&10&1&11&6&2&8&12&5&7&4&3\\\hline
10&11&2&12&7&3&9&1&6&8&5&4\\\hline
11&12&3&1&8&4&10&2&7&9&6&5\\\hline
12&1&4&2&9&5&11&3&8&10&7&6\\\hline
\end{array} \hspace{.3in}
\begin{array}{|c|c|c|c|c|c|c|c|c|c|c|c|}
\hline
1&6&9&11&10&2&8&4&5&3&12&7\\\hline
2&7&10&12&11&3&9&5&6&4&1&8\\\hline
3&8&11&1&12&4&10&6&7&5&2&9\\\hline
4&9&12&2&1&5&11&7&8&6&3&10\\\hline
5&10&1&3&2&6&12&8&9&7&4&11\\\hline
6&11&2&4&3&7&1&9&10&8&5&12\\\hline
7&12&3&5&4&8&2&10&11&9&6&1\\\hline
8&1&4&6&5&9&3&11&12&10&7&2\\\hline
9&2&5&7&6&10&4&12&1&11&8&3\\\hline
10&3&6&8&7&11&5&1&2&12&9&4\\\hline
11&4&7&9&8&12&6&2&3&1&10&5\\\hline
12&5&8&10&9&1&7&3&4&2&11&6\\\hline
\end{array}\]

\scriptsize
\noindent
{\normalsize $n=16$}
\[
\begin{array}{|c|c|c|c|c|c|c|c|c|c|c|c|c|c|c|c|}
\hline
1&3&16&5&12&2&6&7&15&14&10&4&13&8&11&9\\\hline
2&4&1&6&13&3&7&8&16&15&11&5&14&9&12&10\\\hline
3&5&2&7&14&4&8&9&1&16&12&6&15&10&13&11\\\hline
4&6&3&8&15&5&9&10&2&1&13&7&16&11&14&12\\\hline
5&7&4&9&16&6&10&11&3&2&14&8&1&12&15&13\\\hline
6&8&5&10&1&7&11&12&4&3&15&9&2&13&16&14\\\hline
7&9&6&11&2&8&12&13&5&4&16&10&3&14&1&15\\\hline
8&10&7&12&3&9&13&14&6&5&1&11&4&15&2&16\\\hline
9&11&8&13&4&10&14&15&7&6&2&12&5&16&3&1\\\hline
10&12&9&14&5&11&15&16&8&7&3&13&6&1&4&2\\\hline
11&13&10&15&6&12&16&1&9&8&4&14&7&2&5&3\\\hline
12&14&11&16&7&13&1&2&10&9&5&15&8&3&6&4\\\hline
13&15&12&1&8&14&2&3&11&10&6&16&9&4&7&5\\\hline
14&16&13&2&9&15&3&4&12&11&7&1&10&5&8&6\\\hline
15&1&14&3&10&16&4&5&13&12&8&2&11&6&9&7\\\hline
16&2&15&4&11&1&5&6&14&13&9&3&12&7&10&8\\\hline
\end{array}\hspace{.2in}
\begin{array}{|c|c|c|c|c|c|c|c|c|c|c|c|c|c|c|c|}
\hline
1&7&6&13&10&12&8&3&11&16&4&2&5&14&15&9\\\hline
2&8&7&14&11&13&9&4&12&1&5&3&6&15&16&10\\\hline
3&9&8&15&12&14&10&5&13&2&6&4&7&16&1&11\\\hline
4&10&9&16&13&15&11&6&14&3&7&5&8&1&2&12\\\hline
5&11&10&1&14&16&12&7&15&4&8&6&9&2&3&13\\\hline
6&12&11&2&15&1&13&8&16&5&9&7&10&3&4&14\\\hline
7&13&12&3&16&2&14&9&1&6&10&8&11&4&5&15\\\hline
8&14&13&4&1&3&15&10&2&7&11&9&12&5&6&16\\\hline
9&15&14&5&2&4&16&11&3&8&12&10&13&6&7&1\\\hline
10&16&15&6&3&5&1&12&4&9&13&11&14&7&8&2\\\hline
11&1&16&7&4&6&2&13&5&10&14&12&15&8&9&3\\\hline
12&2&1&8&5&7&3&14&6&11&15&13&16&9&10&4\\\hline
13&3&2&9&6&8&4&15&7&12&16&14&1&10&11&5\\\hline
14&4&3&10&7&9&5&16&8&13&1&15&2&11&12&6\\\hline
15&5&4&11&8&10&6&1&9&14&2&16&3&12&13&7\\\hline
16&6&5&12&9&11&7&2&10&15&3&1&4&13&14&8\\\hline
\end{array}\]
\medskip
\[
\begin{array}{|c|c|c|c|c|c|c|c|c|c|c|c|c|c|c|c|}
\hline
1&2&3&4&5&6&7&8&9&10&11&12&13&14&15&16\\\hline
5&10&15&3&8&13&1&6&11&16&4&9&14&2&7&12\\\hline
8&16&7&15&6&14&5&13&4&12&3&11&2&10&1&9\\\hline
6&12&1&7&13&2&8&14&3&9&15&4&10&16&5&11\\\hline
13&9&5&1&14&10&6&2&15&11&7&3&16&12&8&4\\\hline
14&11&8&5&2&16&13&10&7&4&1&15&12&9&6&3\\\hline
2&4&6&8&10&12&14&16&1&3&5&7&9&11&13&15\\\hline
10&3&13&6&16&9&2&12&5&15&8&1&11&4&14&7\\\hline
16&15&14&13&12&11&10&9&8&7&6&5&4&3&2&1\\\hline
12&7&2&14&9&4&16&11&6&1&13&8&3&15&10&5\\\hline
9&1&10&2&11&3&12&4&13&5&14&6&15&7&16&8\\\hline
11&5&16&10&4&15&9&3&14&8&2&13&7&1&12&6\\\hline
4&8&12&16&3&7&11&15&2&6&10&14&1&5&9&13\\\hline
3&6&9&12&15&1&4&7&10&13&16&2&5&8&11&14\\\hline
15&13&11&9&7&5&3&1&16&14&12&10&8&6&4&2\\\hline
7&14&4&11&1&8&15&5&12&2&9&16&6&13&3&10\\\hline
\end{array}\hspace{.2in}
\begin{array}{|c|c|c|c|c|c|c|c|c|c|c|c|c|c|c|c|}
\hline
1&3&8&5&4&10&14&7&15&6&2&12&13&16&11&9\\\hline
2&4&9&6&5&11&15&8&16&7&3&13&14&1&12&10\\\hline
3&5&10&7&6&12&16&9&1&8&4&14&15&2&13&11\\\hline
4&6&11&8&7&13&1&10&2&9&5&15&16&3&14&12\\\hline
5&7&12&9&8&14&2&11&3&10&6&16&1&4&15&13\\\hline
6&8&13&10&9&15&3&12&4&11&7&1&2&5&16&14\\\hline
7&9&14&11&10&16&4&13&5&12&8&2&3&6&1&15\\\hline
8&10&15&12&11&1&5&14&6&13&9&3&4&7&2&16\\\hline
9&11&16&13&12&2&6&15&7&14&10&4&5&8&3&1\\\hline
10&12&1&14&13&3&7&16&8&15&11&5&6&9&4&2\\\hline
11&13&2&15&14&4&8&1&9&16&12&6&7&10&5&3\\\hline
12&14&3&16&15&5&9&2&10&1&13&7&8&11&6&4\\\hline
13&15&4&1&16&6&10&3&11&2&14&8&9&12&7&5\\\hline
14&16&5&2&1&7&11&4&12&3&15&9&10&13&8&6\\\hline
15&1&6&3&2&8&12&5&13&4&16&10&11&14&9&7\\\hline
16&2&7&4&3&9&13&6&14&5&1&11&12&15&10&8\\\hline
\end{array}\]

\bigskip
\renewcommand{\arraycolsep}{.8pt}
\noindent
{\normalsize $n=18$}
\[
\begin{array}{|c|c|c|c|c|c|c|c|c|c|c|c|c|c|c|c|c|c|}
\hline
1&2&14&3&17&15&7&4&9&18&13&16&6&8&12&5&11&10\\\hline
2&3&15&4&18&16&8&5&10&1&14&17&7&9&13&6&12&11\\\hline
3&4&16&5&1&17&9&6&11&2&15&18&8&10&14&7&13&12\\\hline
4&5&17&6&2&18&10&7&12&3&16&1&9&11&15&8&14&13\\\hline
5&6&18&7&3&1&11&8&13&4&17&2&10&12&16&9&15&14\\\hline
6&7&1&8&4&2&12&9&14&5&18&3&11&13&17&10&16&15\\\hline
7&8&2&9&5&3&13&10&15&6&1&4&12&14&18&11&17&16\\\hline
8&9&3&10&6&4&14&11&16&7&2&5&13&15&1&12&18&17\\\hline
9&10&4&11&7&5&15&12&17&8&3&6&14&16&2&13&1&18\\\hline
10&11&5&12&8&6&16&13&18&9&4&7&15&17&3&14&2&1\\\hline
11&12&6&13&9&7&17&14&1&10&5&8&16&18&4&15&3&2\\\hline
12&13&7&14&10&8&18&15&2&11&6&9&17&1&5&16&4&3\\\hline
13&14&8&15&11&9&1&16&3&12&7&10&18&2&6&17&5&4\\\hline
14&15&9&16&12&10&2&17&4&13&8&11&1&3&7&18&6&5\\\hline
15&16&10&17&13&11&3&18&5&14&9&12&2&4&8&1&7&6\\\hline
16&17&11&18&14&12&4&1&6&15&10&13&3&5&9&2&8&7\\\hline
17&18&12&1&15&13&5&2&7&16&11&14&4&6&10&3&9&8\\\hline
18&1&13&2&16&14&6&3&8&17&12&15&5&7&11&4&10&9\\\hline
\end{array}\hspace{.1in}
\begin{array}{|c|c|c|c|c|c|c|c|c|c|c|c|c|c|c|c|c|c|}
\hline
1&2&3&4&5&6&7&8&9&10&11&12&13&14&15&16&17&18\\\hline
3&6&9&12&15&18&2&5&8&11&14&17&1&4&7&10&13&16\\\hline
9&18&8&17&7&16&6&15&5&14&4&13&3&12&2&11&1&10\\\hline
8&16&5&13&2&10&18&7&15&4&12&1&9&17&6&14&3&11\\\hline
5&10&15&1&6&11&16&2&7&12&17&3&8&13&18&4&9&14\\\hline
15&11&7&3&18&14&10&6&2&17&13&9&5&1&16&12&8&4\\\hline
7&14&2&9&16&4&11&18&6&13&1&8&15&3&10&17&5&12\\\hline
2&4&6&8&10&12&14&16&18&1&3&5&7&9&11&13&15&17\\\hline
6&12&18&5&11&17&4&10&16&3&9&15&2&8&14&1&7&13\\\hline
18&17&16&15&14&13&12&11&10&9&8&7&6&5&4&3&2&1\\\hline
16&13&10&7&4&1&17&14&11&8&5&2&18&15&12&9&6&3\\\hline
10&1&11&2&12&3&13&4&14&5&15&6&16&7&17&8&18&9\\\hline
11&3&14&6&17&9&1&12&4&15&7&18&10&2&13&5&16&8\\\hline
14&9&4&18&13&8&3&17&12&7&2&16&11&6&1&15&10&5\\\hline
4&8&12&16&1&5&9&13&17&2&6&10&14&18&3&7&11&15\\\hline
12&5&17&10&3&15&8&1&13&6&18&11&4&16&9&2&14&7\\\hline
17&15&13&11&9&7&5&3&1&18&16&14&12&10&8&6&4&2\\\hline
13&7&1&14&8&2&15&9&3&16&10&4&17&11&5&18&12&6\\\hline
\end{array}\]

\footnotesize
\renewcommand{\arraycolsep}{1.5pt}
\[
\begin{array}{|c|c|c|c|c|c|c|c|c|c|c|c|c|c|c|c|c|c|}
\hline
1&6&12&11&9&17&13&16&5&14&7&4&8&18&2&3&15&10\\\hline
2&7&13&12&10&18&14&17&6&15&8&5&9&1&3&4&16&11\\\hline
3&8&14&13&11&1&15&18&7&16&9&6&10&2&4&5&17&12\\\hline
4&9&15&14&12&2&16&1&8&17&10&7&11&3&5&6&18&13\\\hline
5&10&16&15&13&3&17&2&9&18&11&8&12&4&6&7&1&14\\\hline
6&11&17&16&14&4&18&3&10&1&12&9&13&5&7&8&2&15\\\hline
7&12&18&17&15&5&1&4&11&2&13&10&14&6&8&9&3&16\\\hline
8&13&1&18&16&6&2&5&12&3&14&11&15&7&9&10&4&17\\\hline
9&14&2&1&17&7&3&6&13&4&15&12&16&8&10&11&5&18\\\hline
10&15&3&2&18&8&4&7&14&5&16&13&17&9&11&12&6&1\\\hline
11&16&4&3&1&9&5&8&15&6&17&14&18&10&12&13&7&2\\\hline
12&17&5&4&2&10&6&9&16&7&18&15&1&11&13&14&8&3\\\hline
13&18&6&5&3&11&7&10&17&8&1&16&2&12&14&15&9&4\\\hline
14&1&7&6&4&12&8&11&18&9&2&17&3&13&15&16&10&5\\\hline
15&2&8&7&5&13&9&12&1&10&3&18&4&14&16&17&11&6\\\hline
16&3&9&8&6&14&10&13&2&11&4&1&5&15&17&18&12&7\\\hline
17&4&10&9&7&15&11&14&3&12&5&2&6&16&18&1&13&8\\\hline
18&5&11&10&8&16&12&15&4&13&6&3&7&17&1&2&14&9\\\hline
\end{array}\]

\bigskip \bigskip
\renewcommand{\arraycolsep}{2pt}
\noindent
{\normalsize $n=22$}

\footnotesize
\[
\begin{array}{|c|c|c|c|c|c|c|c|c|c|c|c|c|c|c|c|c|c|c|c|c|c|}
\hline
1&2&3&4&5&6&7&8&9&10&11&12&13&14&15&16&17&18&19&20&21&22\\\hline
7&14&21&5&12&19&3&10&17&1&8&15&22&6&13&20&4&11&18&2&9&16\\\hline
3&6&9&12&15&18&21&1&4&7&10&13&16&19&22&2&5&8&11&14&17&20\\\hline
21&19&17&15&13&11&9&7&5&3&1&22&20&18&16&14&12&10&8&6&4&2\\\hline
9&18&4&13&22&8&17&3&12&21&7&16&2&11&20&6&15&1&10&19&5&14\\\hline
17&11&5&22&16&10&4&21&15&9&3&20&14&8&2&19&13&7&1&18&12&6\\\hline
4&8&12&16&20&1&5&9&13&17&21&2&6&10&14&18&22&3&7&11&15&19\\\hline
5&10&15&20&2&7&12&17&22&4&9&14&19&1&6&11&16&21&3&8&13&18\\\hline
12&1&13&2&14&3&15&4&16&5&17&6&18&7&19&8&20&9&21&10&22&11\\\hline
15&7&22&14&6&21&13&5&20&12&4&19&11&3&18&10&2&17&9&1&16&8\\\hline
13&3&16&6&19&9&22&12&2&15&5&18&8&21&11&1&14&4&17&7&20&10\\\hline
22&21&20&19&18&17&16&15&14&13&12&11&10&9&8&7&6&5&4&3&2&1\\\hline
16&9&2&18&11&4&20&13&6&22&15&8&1&17&10&3&19&12&5&21&14&7\\\hline
20&17&14&11&8&5&2&22&19&16&13&10&7&4&1&21&18&15&12&9&6&3\\\hline
2&4&6&8&10&12&14&16&18&20&22&1&3&5&7&9&11&13&15&17&19&21\\\hline
14&5&19&10&1&15&6&20&11&2&16&7&21&12&3&17&8&22&13&4&18&9\\\hline
6&12&18&1&7&13&19&2&8&14&20&3&9&15&21&4&10&16&22&5&11&17\\\hline
19&15&11&7&3&22&18&14&10&6&2&21&17&13&9&5&1&20&16&12&8&4\\\hline
18&13&8&3&21&16&11&6&1&19&14&9&4&22&17&12&7&2&20&15&10&5\\\hline
11&22&10&21&9&20&8&19&7&18&6&17&5&16&4&15&3&14&2&13&1&12\\\hline
8&16&1&9&17&2&10&18&3&11&19&4&12&20&5&13&21&6&14&22&7&15\\\hline
10&20&7&17&4&14&1&11&21&8&18&5&15&2&12&22&9&19&6&16&3&13\\\hline
\end{array} \]

\[\begin{array}{|c|c|c|c|c|c|c|c|c|c|c|c|c|c|c|c|c|c|c|c|c|c|}\hline
1&3&17&5&2&19&20&7&11&4&10&21&15&22&18&9&8&13&16&6&14&12\\\hline
2&4&18&6&3&20&21&8&12&5&11&22&16&1&19&10&9&14&17&7&15&13\\\hline
3&5&19&7&4&21&22&9&13&6&12&1&17&2&20&11&10&15&18&8&16&14\\\hline
4&6&20&8&5&22&1&10&14&7&13&2&18&3&21&12&11&16&19&9&17&15\\\hline
5&7&21&9&6&1&2&11&15&8&14&3&19&4&22&13&12&17&20&10&18&16\\\hline
6&8&22&10&7&2&3&12&16&9&15&4&20&5&1&14&13&18&21&11&19&17\\\hline
7&9&1&11&8&3&4&13&17&10&16&5&21&6&2&15&14&19&22&12&20&18\\\hline
8&10&2&12&9&4&5&14&18&11&17&6&22&7&3&16&15&20&1&13&21&19\\\hline
9&11&3&13&10&5&6&15&19&12&18&7&1&8&4&17&16&21&2&14&22&20\\\hline
10&12&4&14&11&6&7&16&20&13&19&8&2&9&5&18&17&22&3&15&1&21\\\hline
11&13&5&15&12&7&8&17&21&14&20&9&3&10&6&19&18&1&4&16&2&22\\\hline
12&14&6&16&13&8&9&18&22&15&21&10&4&11&7&20&19&2&5&17&3&1\\\hline
13&15&7&17&14&9&10&19&1&16&22&11&5&12&8&21&20&3&6&18&4&2\\\hline
14&16&8&18&15&10&11&20&2&17&1&12&6&13&9&22&21&4&7&19&5&3\\\hline
15&17&9&19&16&11&12&21&3&18&2&13&7&14&10&1&22&5&8&20&6&4\\\hline
16&18&10&20&17&12&13&22&4&19&3&14&8&15&11&2&1&6&9&21&7&5\\\hline
17&19&11&21&18&13&14&1&5&20&4&15&9&16&12&3&2&7&10&22&8&6\\\hline
18&20&12&22&19&14&15&2&6&21&5&16&10&17&13&4&3&8&11&1&9&7\\\hline
19&21&13&1&20&15&16&3&7&22&6&17&11&18&14&5&4&9&12&2&10&8\\\hline
20&22&14&2&21&16&17&4&8&1&7&18&12&19&15&6&5&10&13&3&11&9\\\hline
21&1&15&3&22&17&18&5&9&2&8&19&13&20&16&7&6&11&14&4&12&10\\\hline
22&2&16&4&1&18&19&6&10&3&9&20&14&21&17&8&7&12&15&5&13&11\\\hline
\end{array} \]
\bigskip

\[\begin{array}{|c|c|c|c|c|c|c|c|c|c|c|c|c|c|c|c|c|c|c|c|c|c|}\hline
1&7&5&13&4&11&14&19&9&10&6&17&21&20&8&3&22&15&2&16&18&12\\\hline
2&8&6&14&5&12&15&20&10&11&7&18&22&21&9&4&1&16&3&17&19&13\\\hline
3&9&7&15&6&13&16&21&11&12&8&19&1&22&10&5&2&17&4&18&20&14\\\hline
4&10&8&16&7&14&17&22&12&13&9&20&2&1&11&6&3&18&5&19&21&15\\\hline
5&11&9&17&8&15&18&1&13&14&10&21&3&2&12&7&4&19&6&20&22&16\\\hline
6&12&10&18&9&16&19&2&14&15&11&22&4&3&13&8&5&20&7&21&1&17\\\hline
7&13&11&19&10&17&20&3&15&16&12&1&5&4&14&9&6&21&8&22&2&18\\\hline
8&14&12&20&11&18&21&4&16&17&13&2&6&5&15&10&7&22&9&1&3&19\\\hline
9&15&13&21&12&19&22&5&17&18&14&3&7&6&16&11&8&1&10&2&4&20\\\hline
10&16&14&22&13&20&1&6&18&19&15&4&8&7&17&12&9&2&11&3&5&21\\\hline
11&17&15&1&14&21&2&7&19&20&16&5&9&8&18&13&10&3&12&4&6&22\\\hline
12&18&16&2&15&22&3&8&20&21&17&6&10&9&19&14&11&4&13&5&7&1\\\hline
13&19&17&3&16&1&4&9&21&22&18&7&11&10&20&15&12&5&14&6&8&2\\\hline
14&20&18&4&17&2&5&10&22&1&19&8&12&11&21&16&13&6&15&7&9&3\\\hline
15&21&19&5&18&3&6&11&1&2&20&9&13&12&22&17&14&7&16&8&10&4\\\hline
16&22&20&6&19&4&7&12&2&3&21&10&14&13&1&18&15&8&17&9&11&5\\\hline
17&1&21&7&20&5&8&13&3&4&22&11&15&14&2&19&16&9&18&10&12&6\\\hline
18&2&22&8&21&6&9&14&4&5&1&12&16&15&3&20&17&10&19&11&13&7\\\hline
19&3&1&9&22&7&10&15&5&6&2&13&17&16&4&21&18&11&20&12&14&8\\\hline
20&4&2&10&1&8&11&16&6&7&3&14&18&17&5&22&19&12&21&13&15&9\\\hline
21&5&3&11&2&9&12&17&7&8&4&15&19&18&6&1&20&13&22&14&16&10\\\hline
22&6&4&12&3&10&13&18&8&9&5&16&20&19&7&2&21&14&1&15&17&11\\\hline
\end{array} \]

\[\begin{array}{|c|c|c|c|c|c|c|c|c|c|c|c|c|c|c|c|c|c|c|c|c|c|}\hline
1&5&11&9&14&15&6&13&21&18&8&19&7&10&2&17&4&3&20&22&16&12\\\hline
2&6&12&10&15&16&7&14&22&19&9&20&8&11&3&18&5&4&21&1&17&13\\\hline
3&7&13&11&16&17&8&15&1&20&10&21&9&12&4&19&6&5&22&2&18&14\\\hline
4&8&14&12&17&18&9&16&2&21&11&22&10&13&5&20&7&6&1&3&19&15\\\hline
5&9&15&13&18&19&10&17&3&22&12&1&11&14&6&21&8&7&2&4&20&16\\\hline
6&10&16&14&19&20&11&18&4&1&13&2&12&15&7&22&9&8&3&5&21&17\\\hline
7&11&17&15&20&21&12&19&5&2&14&3&13&16&8&1&10&9&4&6&22&18\\\hline
8&12&18&16&21&22&13&20&6&3&15&4&14&17&9&2&11&10&5&7&1&19\\\hline
9&13&19&17&22&1&14&21&7&4&16&5&15&18&10&3&12&11&6&8&2&20\\\hline
10&14&20&18&1&2&15&22&8&5&17&6&16&19&11&4&13&12&7&9&3&21\\\hline
11&15&21&19&2&3&16&1&9&6&18&7&17&20&12&5&14&13&8&10&4&22\\\hline
12&16&22&20&3&4&17&2&10&7&19&8&18&21&13&6&15&14&9&11&5&1\\\hline
13&17&1&21&4&5&18&3&11&8&20&9&19&22&14&7&16&15&10&12&6&2\\\hline
14&18&2&22&5&6&19&4&12&9&21&10&20&1&15&8&17&16&11&13&7&3\\\hline
15&19&3&1&6&7&20&5&13&10&22&11&21&2&16&9&18&17&12&14&8&4\\\hline
16&20&4&2&7&8&21&6&14&11&1&12&22&3&17&10&19&18&13&15&9&5\\\hline
17&21&5&3&8&9&22&7&15&12&2&13&1&4&18&11&20&19&14&16&10&6\\\hline
18&22&6&4&9&10&1&8&16&13&3&14&2&5&19&12&21&20&15&17&11&7\\\hline
19&1&7&5&10&11&2&9&17&14&4&15&3&6&20&13&22&21&16&18&12&8\\\hline
20&2&8&6&11&12&3&10&18&15&5&16&4&7&21&14&1&22&17&19&13&9\\\hline
21&3&9&7&12&13&4&11&19&16&6&17&5&8&22&15&2&1&18&20&14&10\\\hline
22&4&10&8&13&14&5&12&20&17&7&18&6&9&1&16&3&2&19&21&15&11\\\hline
\end{array} \]

\bigskip
\[\begin{array}{|c|c|c|c|c|c|c|c|c|c|c|c|c|c|c|c|c|c|c|c|c|c|}\hline
1&11&15&21&6&3&8&9&7&16&2&13&5&18&20&19&14&17&10&4&22&12\\\hline
2&12&16&22&7&4&9&10&8&17&3&14&6&19&21&20&15&18&11&5&1&13\\\hline
3&13&17&1&8&5&10&11&9&18&4&15&7&20&22&21&16&19&12&6&2&14\\\hline
4&14&18&2&9&6&11&12&10&19&5&16&8&21&1&22&17&20&13&7&3&15\\\hline
5&15&19&3&10&7&12&13&11&20&6&17&9&22&2&1&18&21&14&8&4&16\\\hline
6&16&20&4&11&8&13&14&12&21&7&18&10&1&3&2&19&22&15&9&5&17\\\hline
7&17&21&5&12&9&14&15&13&22&8&19&11&2&4&3&20&1&16&10&6&18\\\hline
8&18&22&6&13&10&15&16&14&1&9&20&12&3&5&4&21&2&17&11&7&19\\\hline
9&19&1&7&14&11&16&17&15&2&10&21&13&4&6&5&22&3&18&12&8&20\\\hline
10&20&2&8&15&12&17&18&16&3&11&22&14&5&7&6&1&4&19&13&9&21\\\hline
11&21&3&9&16&13&18&19&17&4&12&1&15&6&8&7&2&5&20&14&10&22\\\hline
12&22&4&10&17&14&19&20&18&5&13&2&16&7&9&8&3&6&21&15&11&1\\\hline
13&1&5&11&18&15&20&21&19&6&14&3&17&8&10&9&4&7&22&16&12&2\\\hline
14&2&6&12&19&16&21&22&20&7&15&4&18&9&11&10&5&8&1&17&13&3\\\hline
15&3&7&13&20&17&22&1&21&8&16&5&19&10&12&11&6&9&2&18&14&4\\\hline
16&4&8&14&21&18&1&2&22&9&17&6&20&11&13&12&7&10&3&19&15&5\\\hline
17&5&9&15&22&19&2&3&1&10&18&7&21&12&14&13&8&11&4&20&16&6\\\hline
18&6&10&16&1&20&3&4&2&11&19&8&22&13&15&14&9&12&5&21&17&7\\\hline
19&7&11&17&2&21&4&5&3&12&20&9&1&14&16&15&10&13&6&22&18&8\\\hline
20&8&12&18&3&22&5&6&4&13&21&10&2&15&17&16&11&14&7&1&19&9\\\hline
21&9&13&19&4&1&6&7&5&14&22&11&3&16&18&17&12&15&8&2&20&10\\\hline
22&10&14&20&5&2&7&8&6&15&1&12&4&17&19&18&13&16&9&3&21&11\\\hline
\end{array} \]

\end{document}